\documentclass[11pt,a4paper]{article}
\usepackage[dvips]{graphicx}
\usepackage[ansinew]{inputenc}
\usepackage{enumerate}
\usepackage[spanish,english]{babel}
\usepackage{amsthm,amssymb,amsfonts,amsmath}
\usepackage[dvips,dvipsnames,usenames]{color}
\newtheorem{theorem}{Theorem}[section]
\newtheorem{lemma}[theorem]{Lemma}
\newtheorem{conjecture}[theorem]{Conjecture}
\newtheorem{remark}[theorem]{Remark}
\newtheorem{proposition}[theorem]{Proposition}

\newtheorem{problem}[theorem]{Problem}
\DeclareMathOperator{\res}{res} \DeclareMathOperator{\dime}{dim}
  \DeclareMathOperator{\ecc}{ecc}
 \DeclareMathOperator{\g}{g}

\title{On the metric dimension, the upper dimension and the resolving number of graphs}

\author{Delia Garijo\thanks{Department of
Applied Mathematics I, University of Seville, Spain. Email
addresses: {\tt \{dgarijo,gonzalezh,almar\}@us.es}.} \and Antonio
González$^*$ \and Alberto Márquez$^*$}
\date{}
\begin{document}
\maketitle

\begin{abstract}
This paper deals with three resolving parameters: the metric
dimension, the upper dimension and the resolving number. We first
answer a question raised by Chartrand and Zhang asking for a
characterization of the graphs with equal metric dimension and
resolving number. We also solve in the affirmative a conjecture
posed by Chartrand, Poisson and Zhang about the realization of
the metric dimension and the upper dimension. Finally we prove
that no integer $a\geq 4$ is realizable as the resolving number
of an infinite family of graphs.

\end{abstract}

\section{Introduction}

In this paper, we study resolving sets for finite simple
connected graphs. They were introduced in the 1970s independently
by Slater \cite{slater}, and Harary and Melter \cite{melter}. The
usefulness of these sets comes from their multiple applications
in several areas, among them: coin weighing problems, network
discovery and verification, robot navigation, strategies for
Mastermind game and chemical industry (we refer the reader to
\cite{cartesian}  for a number of references on this topic).
Resolving sets are formally defined as follows.

Let $G=(V(G),E(G))$ be a finite, simple, connected graph of order
$n=|V(G)|$. The \emph{distance} $d(u,v)$ between two vertices
$u,v\in V(G)$ is the length of a shortest $u$-$v$ path in $G$. A
vertex $u\in V(G)$ \emph{resolves} a pair $\{x,y\}\subset V(G)$
if $d(u,x)\neq d(u,y)$. A set of vertices $S\subseteq V(G)$ is a
\emph{resolving set} of $G$ if every pair of vertices of $G$ is
resolved by some vertex of $S$. A resolving set $S$ of minimum
size is a  \emph{metric basis}, and $|S|$ is the \emph{metric
dimension} of $G$, denoted by $\dime (G)$.

Our aim is not only to deal with metric bases and metric dimension
but also with two other resolving parameters defined by Chartrand
et al. \cite{upper}, \emph{the upper dimension} and \emph{the
resolving number}, that give an insight of how dense the set of
resolving sets of a graph is.

A resolving set $S$ of $G$ is \emph{minimal} if no proper subset
of $S$ is a resolving set. An \emph{upper basis}  is a minimal
resolving set containing the maximum number of vertices. The
\emph{upper dimension} ${\rm dim}^+(G)$ is the
 size of an upper basis. The \emph{resolving number} ${\rm res}(G)$  is the
minimum $k$ such that every $k$-subset of $V(G)$ is a resolving
set of $G$.

Clearly, every $(n - 1)$-subset of $V(G)$ is a resolving set  and
every resolving set contains a minimal resolving set. Hence,
$$1\leq {\rm dim}(G) \leq {\rm dim}^+(G) \leq {\rm res}(G) \leq n
- 1.$$

When $\dime (G)=\res (G)=k$, the graph $G$ is called
\emph{randomly $k$-dimensional}, that is, every subset of size
$k$ is a metric basis and so $G$ has the maximum number of metric
bases.

Chartrand and Zhang \cite{chrom} posed the problem of
characterizing the randomly $k$-dimensional graphs. They solved
the case $k\leq 2$, obtaining the complete graphs $K_1$ and $K_2$
(for $k=1$) and odd cycles (for $k=2$) which includes as a
particular case the complete graph $K_3$. Nevertheless, it
remained open the main following question.

\begin{problem}\label{problem}{\rm\cite{chrom}}
Are there randomly $k$-dimensional graphs other than complete
graphs and odd cycles?
\end{problem}

Concerning the three parameters, Chartrand et al. \cite{upper}
investigated some relationships among them. They proved that
every pair $a,b$ of integers with $2\leq a \leq b$ is realizable
as the metric dimension and the resolving number, respectively, of
some connected graph $G$. It was also shown the analogous result
for ${\rm dim}(G)={\rm dim}^+(G) =a$ and $\res(G)= b$. Moreover,
the authors proved that every pair among the three parameters  can
differ by an arbitrarily large number. Thus, as remarked in
\cite{upper}, there was reason to believe that every pair $a,b$
of integers with $2\leq a \leq b$ is realizable as the metric
dimension and the upper dimension, respectively, of some
connected graph. This remained as a conjecture.

\begin{conjecture}\label{conjecture}{\rm\cite{upper}}
For every pair $a, b$ of integers with $2 \leq a \leq b$, there
exists a connected graph $G$ with ${\rm dim}(G) = a$ and ${\rm
dim}^+(G) = b$.
\end{conjecture}

In this paper, we provide a combinatorial proof avoiding the brute
force casuistic analysis to solve Problem~\ref{problem} (see
Theorem \ref{charact}). We also
 prove in the affirmative
Conjecture~\ref{conjecture} (see Theorem \ref{ab}) and show that
no integer $a\geq 4$ is realizable as the resolving number of an
infinite family of graphs (see Theorem~\ref{resfinite}).

\section{Graphs $G$ with $\dime (G)=\res (G)$.}

In this section, we characterize the randomly $k$-dimensional
graphs. The main difficulty in this problem is to avoid the
casuistic analysis resulting from the intuitive idea of
considering metric bases formed by $k$ vertices close to a given
vertex, and then to try to extend this local argument to the
whole graph. In fact, while preparing this paper, we have learnt
of \cite{ckrandom}, where the authors prove the same result going
through this type of analysis and using also results related to
the degree of the vertices, the connectivity and the induced
subgraphs. Here, we present an alternative  proof based on
combinatorial arguments.

We start with some technical lemmas needed for the case $k=3$.
Denote by $\mathcal{P}_{\lambda}(V(G))$ the $\lambda$-subsets of
$V(G)$ and let $N_i(u)$ be the set of vertices at distance $i$
from $u\in V(G)$. For $\{u,v\}, \{x,y\}\in \mathcal{P}_2(V(G))$
we say that \emph{the pair} $\{u,v\}$ \emph{resolves the pair}
$\{x,y\}$ if either $u$ or $v$ resolves it.

\begin{lemma}\label{keylemma}
Let $G$ be a randomly $3$-dimensional graph. Then, the following
statements hold.
\begin{enumerate}
\item[{\rm (a)}] For every pair $\{u,v\}\in \mathcal{P}_2(V(G))$ there exist unique pairs $\{x,y\},\{r,s\}\in \mathcal{P}_2(V(G))$
such that $\{x,y\}$ is not resolved by $\{u,v\}$, and $\{u,v\}$
is not resolved by $\{r,s\}$.
%\item[{\rm (b)}]  The number of non-resolved pairs of vertices by any vertex $u\in V(G)$ is equal to $n-1$.

\item[{\rm (b)}]  Every vertex $u\in V(G)$ satisfies that
\begin{equation}\label{eq1}\sum_{1\leq i\leq \ecc(u)}
\binom{|N_i(u)|}{2}=n-1 \end{equation} where $\ecc (u)$ denotes
the eccentricity of $u$, i.e., the maximum distance from $u$ to
any other vertex.
\end{enumerate}
\end{lemma}

\begin{proof}

To prove Statement (a), consider a graph $G$ verifying that
$\dime(G) = \res(G) = 3$. Clearly, for every pair of vertices
$\{u, v\}$ there is a pair $\{x, y\}$ that is not resolved by
$\{u, v\}$ (otherwise $\dime(G) = 2)$. Suppose on the contrary
that there are two pairs $\{u, v\},\{\tilde{u}, \tilde{v}\}$ that
have the same associated pair $\{x, y\}\in\mathcal{P}_2(V (G))$,
i.e., $\{x, y\}$ is not resolved by either $\{u, v\}$ or
$\{\tilde{u}, \tilde{v}\}$, and assume that $u, v\neq \tilde{u}$.
Then the set $\{u, v, \tilde{u}\}$ is not a metric basis, which
is a contradiction.  It is analogous to prove that there is a
unique pair $\{r,s\}\in \mathcal{P}_2(V(G))$ such that $\{u,v\}$
is not resolved by $\{r,s\}$. Hence, the result follows.

%Statement (b) is a consequence of Statement (a), taking $u\in
%V(G)$ and considering the $n-1$ distinct pairs $\{u,v\}$ with
%$v\in V(G)\setminus\{u\}$.

As a consequence of Statement (a) we have that the size of the
set of non-resolved pairs  by a vertex $u\in V(G)$ is equal to
$n-1$ (it suffices to consider the $n-1$ distinct pairs $\{u,v\}$
with $v\in V(G)\setminus\{u\}$). This set is formed by pairs of
vertices at the same distance from $u$ and so its size is equal
to $\displaystyle{\sum_{1\leq i\leq
\ecc(u)}}\binom{|N_i(u)|}{2}$, which proves Statement (b).

\end{proof}

\begin{remark}\label{remark} Fixed $u\in V(G)$, consider the  partition $\mathcal{P}(u)=\{N_i(u) \, | \, 0\leq
i\leq \ecc(u)\}$ of $V(G)$ into classes (where $N_{0}(u)=\{u\}$).
Lemma~\ref{keylemma}(b) says that there is a compensation between
vertices of $V(G)\setminus \{u\}$ and
 pairs of vertices located in the same class of
$\mathcal{P}(u)$. For instance, classes of size at least $4$
always contribute to Equation \ref{eq1} with more pairs than
vertices ($6$ pairs and $4$ vertices in case of size $4$) and so
they have to be compensated with classes of size at most $2$ whose
contribution is bigger in terms of vertices than in pairs. Note
that classes of size $3$ which contribute with $3$ pairs, are
self-compensated. Therefore, the existence of a class of size at
least $4$ in the partition $\mathcal{P}(u)$ is equivalent to the
existence of at least two classes of size at most $2$.
\end{remark}

The following straightforward lemma will be useful for the proofs
of this paper.

\begin{lemma}\label{landmarks} {\rm\cite{landmarks}} Let $u,v,w\in V(G)$ such that $\{v,w\}\in E(G)$ and $d(u,v)=d$.
Then $d(u,w)\in\{d-1,d,d+1\}$.
\end{lemma}

Lemma \ref{keylemma} and Remark \ref{remark} are the key tools to
prove the following lemma which let us avoid the casuistic
analysis to characterize the randomly $3$-dimensional graphs.

\begin{lemma}\label{regularity}
Let $G$ be a randomly $3$-dimensional graph of order $n$. Then,
$G$ is $3$-regular and $n\in\{4,7,10\}$.
\end{lemma}

\begin{proof}
First, observe that $G$ does not contain vertices of degree $1$.
Indeed, if a vertex $u$ has a unique neighbour $v$, then the pair
$\{u,v\}$ is resolved by every vertex of $G$, which contradicts
Lemma~\ref{keylemma}(a).

\

\noindent \emph{Claim 1. The degree of every vertex of $G$ is at
most 3.}

\begin{proof}
Suppose on the contrary that there exists $u\in V(G)$ of degree at
least 4 and let $ u_1,u_2,u_3,u_4\in N_1(u)$. By
Lemma~\ref{keylemma}(a), each set $\mathcal{A}_{ij}=\{v\in V(G)
\, | \, d(v,u_i)=d(v,u_j)\}$ with $1\leq i<j\leq 4$ contains
exactly two vertices of $G$. Moreover, Lemma~\ref{landmarks}
implies that every vertex of $G$ belongs to at least one of the
six sets $\mathcal{A}_{ij}$. Hence, $n\leq 7$ since  $u\in
\mathcal{A}_{ij}$ for all $i,j$.

Consider now the partition $\mathcal{P}(u)$ in which $N_1(u)$ is a
class of size at least 4. By Remark~\ref{remark}, there are at
least two classes of size at most $2$. Even more, since $n\leq 7$
then there are  exactly two classes of size 1 and so the furthest
vertex from $u$ has degree $1$; a contradiction. Therefore, every
vertex of $G$ has degree at most 3.
\end{proof}

\,

\noindent \emph{Claim 2. $n\in \{4,7,10\}$}.

\begin{proof}
Since $\dime(G)=3$ then $G$ is neither a path nor a cycle and so
there is a vertex $u\in V(G)$ of degree $3$ with neighbours, say
$u_1,u_2,u_3$. Arguing as in the proof of Claim 1 (defining the
analogous sets $\mathcal{A}_{ij}$ but for the vertices
$u,u_1,u_2,u_3$) we have that $n\leq 10$.

The sets $\{u\}$ and $\{u_1,u_2,u_3\}$ are the classes $N_0(u)$
and $N_1(u)$, respectively, in the partition $\mathcal{P}(u)$. If
this partition does not contain more classes, then $n=4$.
Otherwise, Remark~\ref{remark} says that the existence of a class
of size at least $4$ is equivalent to the existence of at least
two classes of size at most $2$, and classes of size 3 are
self-compensated. Since $n\leq 10$, it is easy to check that
there are two possibilities for $\mathcal{P}(u)$: (1) a class of
size 4 and two classes of size 1 (plus $N_0(u)$ and $N_1(u)$);
(2) two classes of size 3 (one being $N_1(u)$). This gives $n=10$
and $n=7$, respectively.
\end{proof}

\,

\noindent \emph{Claim 3. There is no vertex of degree 2.}

\begin{proof}
Suppose on the contrary that there is a vertex $u\in V(G)$ of
degree 2. Then, $|N_1(u)|=2$. By Remark~\ref{remark},
$\mathcal{P}(u)$ contains a class of size at least 4 and another
class of size at most 2.  Since $n\leq 10$  we have the following
two possibilities for $\mathcal{P}(u)$: (1) one class of size 4,
two classes of size 1 and one class of size 2; (2) one class of
size 4, one class of size 1 and two classes of size 2. This gives,
respectively, $n=8$ and $n=9$ which contradicts Claim 2.
\end{proof}
%
%Therefore, every vertex of $G$ has degree $3$ and $n\in\{4,7,10\}$, the
%desired result.

The three previous claims prove that a graph $G$ of order $n$
satisfying $\dime(G) = \res(G) = 3$  is $3$-regular and
$n\in\{4,7,10\}$.

\end{proof}

Now, we reach the desired characterization that solves Problem
\ref{problem}.

\begin{theorem}\label{charact}
A graph $G$ is randomly $k$-dimensional if and only if $G$ is a
complete graph or an odd cycle.

%$\dime(G)=\res(G)=k$, $k\geq 3$ if and only if  $G\cong K_{k+1}$.
\end{theorem}

\begin{proof}
If $G$ is isomorphic to a complete graph or an odd cycle, it is
straightforward to prove that $G$ is randomly $k$-dimensional.

Suppose now that $G$ is a graph of order $n$ satisfying
$\dime(G)=\res(G)=$ $k$. We can assume $k\geq 3$ (as it was said
before the case $k\leq 2$ is proved in \cite{chrom}, obtaining
the complete graphs $K_1$ and $K_2$ (for $k=1$) and odd cycles
(for $k=2$)). Suppose first that $k=3$ and so
$\dime(G)=\res(G)=3$. By Lemma~\ref{regularity}, $G$ is 3-regular
and $n\in\{4,7,10\}$. We shall now prove that $n=4$.

Clearly, $n\neq 7$ since there is no 3-regular graph with 7
vertices. Consider now two vertices
 $u, v \in V(G)$ so that $d(u,v)=d(G)$, where $d(G)$
denotes the diameter of $G$. Let $N_1(u)=\{u_1,u_2,u_3\}$. By
Lemma~\ref{landmarks}, the distance from $v$ to every vertex of
the set $\{u,u_1,u_2,u_3\}$ is either $d(G)$ or $d(G)-1$. Hence,
$v$ belongs to at least two of the sets $\mathcal{A}_{ij}$
defined as in the proof of Claim 1 but for the vertices
$u,u_1,u_2,u_3$. By Lemma \ref{keylemma}(a), each set contains
exactly two vertices of $G$ and $u$ belongs to three of them.
This gives $n<10$ and so
 $n=4$ which implies that $G$ is isomorphic to $K_4$ (the only
3-regular graph with 4 vertices).

Suppose now that $k\geq 4$ and assume $\dime(G)=\res(G)=$ $k$.
Arguing as in the proof of Lemma~\ref{keylemma}(a) we have that
for every $T\in {\cal{P}}_{k-1}(V(G))$, the non-empty set
$$S_T=\{\{x,y\}\, | \, T \,  {\rm does \, not \, resolve} \,
\{x,y\}\}\subset {\cal{P}}_{2}(V(G))$$ verifies that $S_T \cap
S_{T'}=\emptyset$ whenever $T\neq T'$. Therefore
$|{\cal{P}}_{k-1}(V(G))|\leq |{\cal{P}}_{2}(V(G))|$, i.e.,
$$\binom{n}{k-1}\leq \binom{n}{2} \Longrightarrow k\in
\{1,2,3,n-1,n,n+1\}.$$ Hence $k=n-1$ since $4\leq k=\dim(G)$. This
 implies that $G$ is isomorphic to the complete
graph $K_{n}$ which is the only graph verifying that
$\dime(G)=k=n-1$.
\end{proof}

\section{Realization of graphs}

\subsection{Realization of the metric dimension and the upper
dimension}

This subsection is devoted to prove in the affirmative
Conjecture~\ref{conjecture}. In order to do this, we  compute the
upper dimension of two families of graphs for which the metric
dimension is easily obtained. These graphs are constructed from
the grid graphs attaching at the origin either a triangle or a
number of pendant vertices. We start with some notation and
technical lemmas.

Let $G_{\ell}$ be the $2$-dimensional grid graph of size
$\ell\times \ell$ with $\ell\geq 2$, whose vertex set is the
cartesian product $[0,\ell-1]\times [0,\ell-1]$ and distances
given by $d((x_1,x_2),(y_1,y_2))=|x_1-y_1|+|x_2-y_2|$. We shall
use $x_1,x_2$ to indicate the coordinates of a vertex $x\in
V(G_{\ell})$ (analogously $y=(y_1,y_2)$, $z=(z_1,z_2)$,
$r=(r_1,r_2)$, etc.) The following sets of vertices are called
\emph{quadrants of} $x\in V(G_{\ell})$:  $$Q_1(x)=\{y\in
V(G_{\ell}) \, | \, y_1\geq x_1,y_2\geq x_2\}, \, Q_2(x)=\{y\in
V(G_{\ell}) \, | \, y_1\leq x_1,y_2\geq x_2\},$$ $$Q_3(x)=\{y\in
V(G_{\ell})\, | \, y_1\leq x_1,y_2\leq x_2\}, \, Q_4(x)=\{y\in
V(G_{\ell})\, | \, y_1\geq x_1,y_2\leq x_2\},$$ and the  sets
$D_i=\{x\in V(G_{\ell})\, | \, x_1+x_2=i\}$ for $0\leq i \leq
\ell-1$ are the {\it diagonals} of $G_{\ell}$ (see Figure
\ref{quadrants}(a)). A pair of vertices $\{x,y\}$ is said to be a
\emph{diagonal pair} if $x,y\in D_i$ for some $i$. Note that a
quadrant $Q_i(x)$ might be equal to $\{x\}$ and there is a total
order $<_i$ in each diagonal $D_i$ (or simply ''$<$" when no
confusion can arise) given by
\begin{center}
\begin{tabular}{ccc}
$x<_i y$& $\Longleftrightarrow$&  $x_1<y_1$.
\end{tabular}
\end{center}
In the sequel, we shall assume without loss of generality that
the order of the two elements of a diagonal pair $\{x,y\}$ is
$x<y$ (analogously $r<s$ for $\{r,s\}$ or $t<z$ for $\{t,z\}$).

Let $R(x,y)$ be the set of vertices of $G_{\ell}$ that resolve the
pair $\{x,y\} \subset V(G_{\ell})$, and let $S$ be a resolving set
of $G_{\ell}$. Note that the set $R(x,y)\cap S$ is non-empty for
every pair $\{x,y\}$.

\begin{lemma}\label{lemma1}
Let $\{x,y\}$ be a diagonal pair such that $d(x,y)=2$. Then,
$$R(x,y)=Q_2(x)\cup Q_4(y).$$
\end{lemma}
\begin{proof}
 Every vertex $u\in
Q_2(x)$ has a shortest $u$-$y$ path through $x$  and so
$d(u,y)=d(u,x)+d(x,y)=d(u,x)+2$. Thus, $u$ resolves $\{x,y\}$
(analogous for $u\in Q_4(y)$).

Let $u\in V(G_{\ell})\setminus(Q_2(x)\cup Q_4(y))$, $z=(x_1,y_2)$
and $\tilde{z}=(y_1,x_2)$. Clearly, there are two shortest paths
$P_1, P_2$ joining $u$ to $x$ and $u$ to $y$, respectively, such
that either $z\in P_1, P_2$ or $\tilde{z}\in P_1, P_2$ (see
Figure \ref{quadrants}(b)). Since $z,\tilde{z}$ do not resolve
the pair $\{x,y\}$ then $u \notin R(x,y)$ .
\end{proof}

A pair $\{x,y\}$ is said to be \emph{$S$-unique} if there is a
unique vertex $u \in S$ resolving $\{x,y\}$, i.e., $R(x,y)\cap
S=\{u\}$. The vertex $u$ is called the \emph{associated vertex of
the pair} $\{x,y\}$.  The following remark is straightforward.

\begin{figure}
\begin{center}
\begin{tabular}{ccc}
\includegraphics[width=0.4\textwidth]{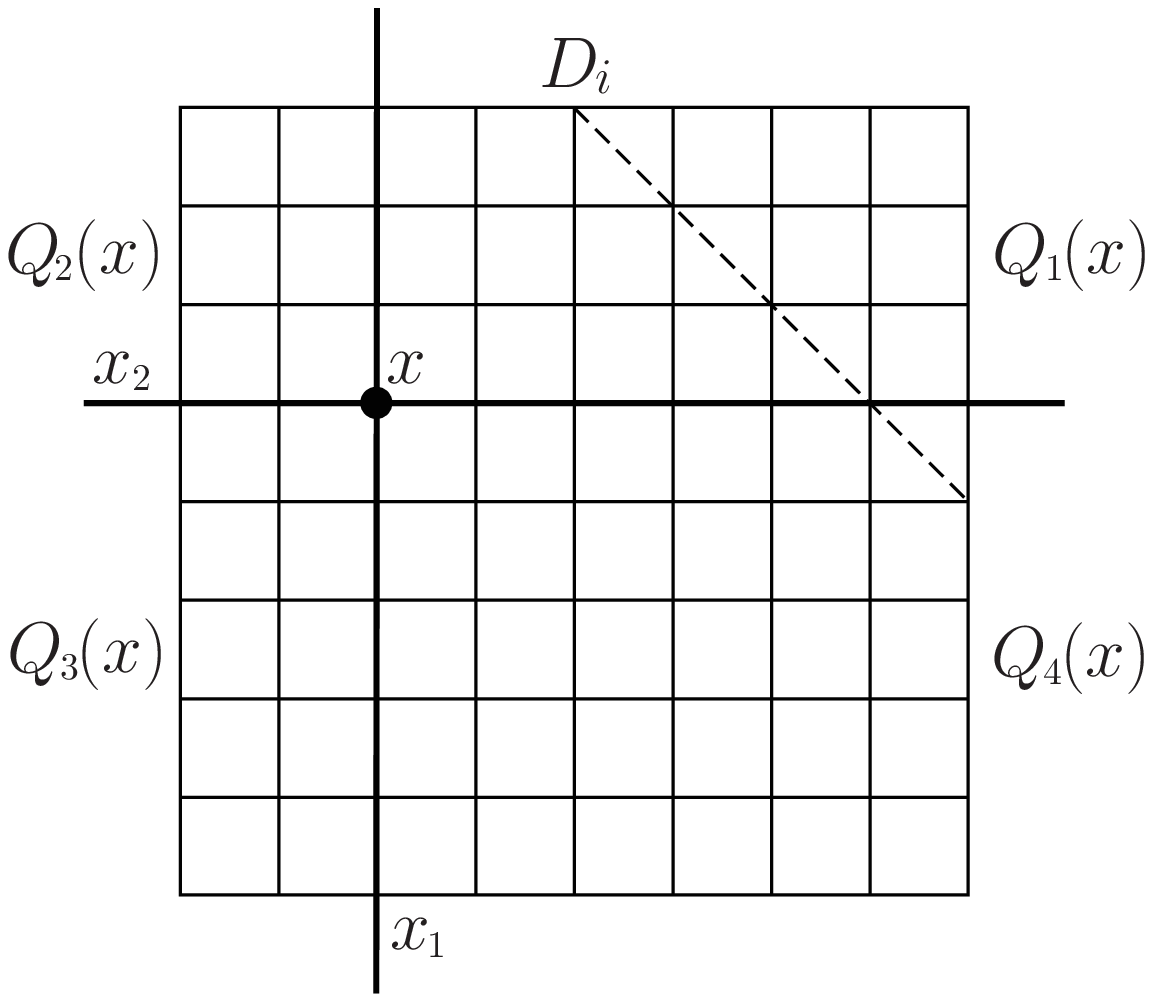}
&&
\includegraphics[width=0.4\textwidth]{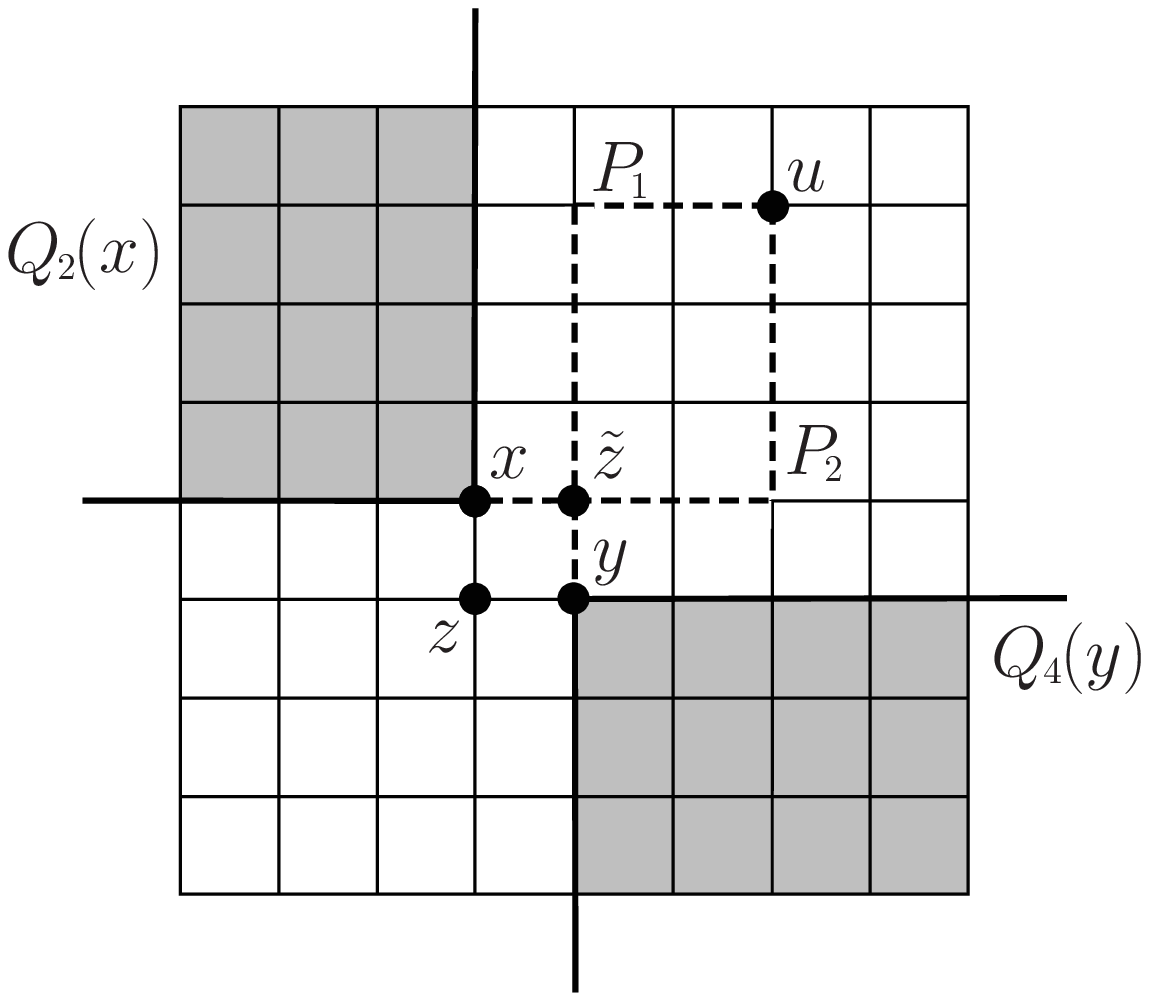}\\
(a) &   &  (b) \\
\end{tabular}
\end{center}
\caption{(a) Quadrants of $x$ and a diagonal $D_i$, (b) The
shadowed region is $R(x,y)=Q_2(x)\cup Q_4(y)$. The dotted edges
form the two shortest paths $P_1, P_2$.}\label{quadrants}
\end{figure}

\begin{remark} \label{contencion} Let $\{x,y\}$ be an $S$-unique pair with associated vertex $u$. If there is a pair $\{r,s\}$
such that $R(r,s)\subseteq R(x,y)$ then $\{r,s\}$ is $S$-unique
with associated vertex $u$.

\end{remark}

\begin{lemma}\label{lemma2}
Let $\{x,y\}=\{(x_1,x_2),(y_1,y_2)\}$ be an $S$-unique diagonal
pair with associated vertex $u$ such that $d(x,y)>2$. Then there
exist $y_1-x_1$ $S$-unique diagonal pairs $\{r,s\}$ with
associated vertex $u$ and $d(r,s)=2$.
\end{lemma}

\begin{proof}
A similar argument as in the proof of Lemma~\ref{lemma1},
considering $z=(x_1,y_2)$ and $\tilde{z}=(y_1,x_2)$ gives that
every vertex $u\in Q_3(z)\cup Q_1(\tilde{z})$ does not resolve
the pair $\{x,y\}$. We have to add the vertices $(x_1+j, \,
y_2+j)$ with $0< j < y_1-x_1$ which clearly do not resolve the
pair $\{x,y\}$ either (see Figure \ref{pairs}(a)). Thus, the
expression of $R(x,y)$ for vertices at distance bigger than 2  is
$$R(x,y)=V(G_{\ell})\setminus (Q_3(z)\cup Q_1(\tilde{z})\cup
\{(x_1+j,y_2+j) \, | \, 0< j < y_1-x_1\}).$$ This set can also be
expressed as follows: $$R(x,y)=\bigcup_{\small{0\leq j <
y_1-x_1}} R(r^j,s^j)$$ where $r^j=(x_1+j,y_2+j+1)$,
$s^j=(x_1+j+1,y_2+j)$ and $d(r^j,s^j)=2$ (see Figure
\ref{pairs}(b)). Since $R(r^j,s^j) \subseteq R(x,y)$, by Remark
\ref{contencion}, the result holds.
\end{proof}

Two diagonal pairs $\{x,y\}$, $\{r,s\}$ with $d(x,y)=d(r,s)=2$ are
said to be \emph{in the same row} if $x_2=r_2$ and $y_2=s_2$.
Analogously, they are \emph{in the same column} if $x_1=r_1$ and
$y_1=s_1$.

\begin{figure}
\begin{center}
\begin{tabular}{ccc}
\includegraphics[width=0.4\textwidth]{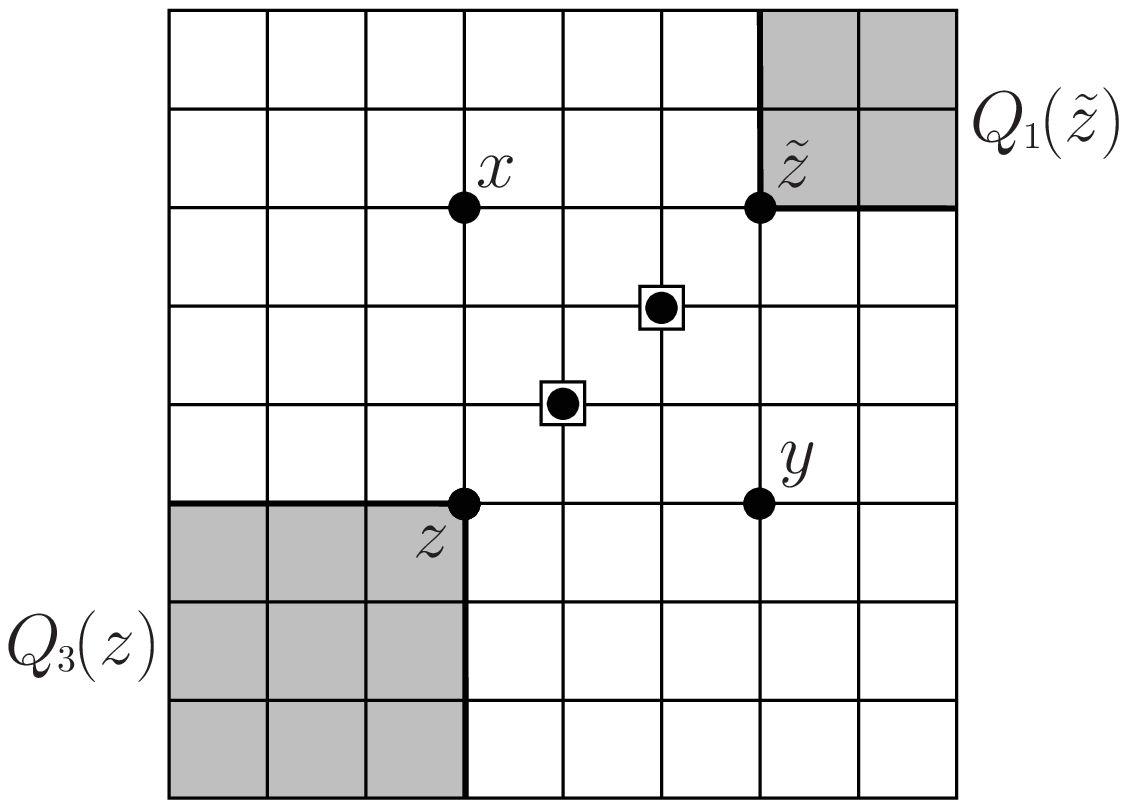}
&&
\includegraphics[width=0.28\textwidth]{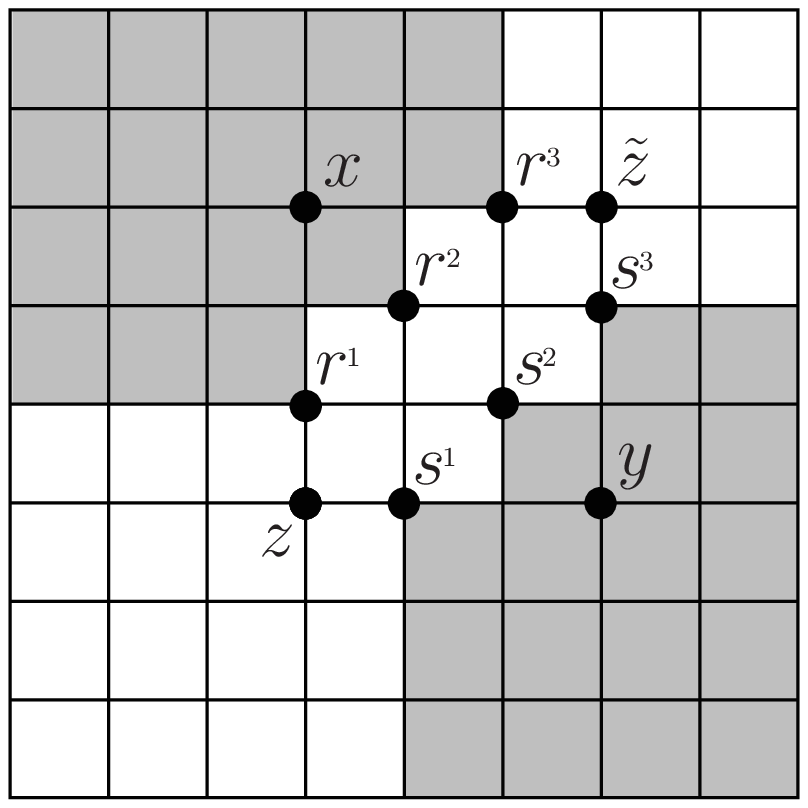}\\
(a) &   &  (b) \\
\end{tabular}
\end{center}
\caption{(a) All the vertices in the shadowed region plus the two
squared vertices do not resolve the pair $\{x,y\}$, (b) The
shadowed region illustrates $R(x,y)$.}\label{pairs}
\end{figure}

\begin{lemma}\label{lemma3}
 Let $\{x,y\}$ be an
$S$-unique diagonal pair with associated vertex $u$ such that
$d(x,y)=2$. If there exist two $S$-unique diagonal pairs
$\{r,s\}$, $\{t,z\}$ in the same row (column) than $\{x,y\}$ with
associated vertices, respectively, $v$ and $w$ and $u\neq v,w$
then $v=w$.
\end{lemma}

\begin{proof}
Suppose that the pairs $\{r,s\}$, $\{t,z\}$ are in the same row
(analogous for columns) than $\{x,y\}$, i.e., $x_2=r_2=t_2$ and
$y_2=s_2=z_2$. Assume also that $x_1<r_1 < t_1$. Clearly,
 $R(r,s)\subset R(x,y) \cup R(t,z)$  and so  $v=w$ since $v\neq u$.
\end{proof}

Now, we reach our main result in this subsection which answers in
the affirmative Conjecture \ref{conjecture}.

\begin{theorem}\label{ab}
For every pair $a, b$ of integers with $2 \leq a \leq b$, there exists a connected
graph $G$ with $\dime(G) = a$ and $\dime^+(G) = b$.
\end{theorem}

\begin{proof}

Let $H_{\ell}$ be the graph obtained from $G_{\ell}$, $\ell\geq
2$, by attaching a triangle at the vertex $(0,0)$, i.e.,
 $V(H_{\ell})=V(G_{\ell})\cup \{\alpha,\beta\}$ and
$E(H_{\ell})=E(G_{\ell})\cup
\{\{\alpha,\beta\},\{\alpha,(0,0)\},\{\beta,(0,0)\}\}$. Observe
that distances in $H_{\ell}$ behave as in $G_{\ell}$, except for
the new vertices $\alpha$ and $\beta$ for which
$d(\alpha,x)=d(\beta,x)=x_1+x_2+1$ for every $x=(x_1,x_2)\in
V(G_{\ell})$. Thus, the previous lemmas can be applied to the
graph $H_{\ell}$.

\

\noindent \emph{Claim 1. $\dime(H_{\ell})=2$ and
$\dime^+(H_{\ell})=2\ell-2$.}

\begin{proof}
It is well-known that $\dime (G_{\ell})=2$ being the set
$\{(0,0),(\ell-1,0)\}$ a metric basis (see for instance
\cite{landmarks}). This set can be adapted to a metric basis of
$H_{\ell}$ by considering $\{\alpha,(\ell-1,0)\}$. Hence,
$\dime(H_{\ell})=2$.

To prove that $\dim ^+(H_{\ell})\geq 2\ell-2$ one can easily
check that the set $$S=\{(x_1,x_2) \, | \, 1\leq x_1\leq \ell-2,
\, x_2=x_1,x_1+1\} \cup \{(0,1),\alpha\}$$ is a resolving set of
$H_{\ell}$ of size $2(\ell-1)$. Moreover, $S$ is minimal because
removing either a vertex $(x_1,x_1)$ or $(x_1,x_1+1)$ from $S$
gives that either the pair $\{(x_1,x_1), (x_1-1,x_1+1) \}$ or the
pair $\{(x_1,x_1+1), (x_1+1,x_1) \}$ is not resolved by any
element of $S$. Clearly, $(0,1)$ and $\alpha$ cannot be removed
from $S$. Figure \ref{metricbasis}(a) illustrates this minimal
resolving set of $H_{\ell}$.

We next prove that $\dim ^+(H_{\ell})\leq 2\ell-2$. Let $S$ be a
minimal resolving set of $H_{\ell}$ and consider the pair
$\{\alpha,\beta\}$ which is only resolved by either $\alpha$ or
$\beta$ and so we can assume that $\alpha\in S$ (otherwise
$\beta\in S$).

Since $S$ is minimal, every vertex $u\in S$ has an associated
$S$-unique pair, say $p(u)$. Observe that $\{\beta,(0,0)\}$ is
not an $S$-unique pair (every vertex of $G_{\ell}$ resolves it)
and so there is no vertex $u\in S$ so that
$p(u)=\{\beta,(0,0)\}$. Note also that $\alpha$ resolves all the
non-diagonal pairs of $G_{\ell}$. Hence, every vertex $u\in
S\setminus \{\alpha\}$ has an associated $S$-unique diagonal pair
$p(u)$. Moreover, by Lemma \ref{lemma2}, we can assume that the
elements of $p(u)$ are at distance 2. Thus, Lemma~\ref{lemma3}
says that $|S\setminus \{\alpha\}|\leq 2(\ell-1)$ and so we still
need to reduce the bound in one unit.

By Lemma \ref{lemma1}, $R((0,1),(1,0))=Q_2((0,1))\cup Q_4((1,0))=
\{(0,x_2)\, | \, 1 \leq x_2 \leq \ell-1\} \cup \{(x_1,0) \, | \, 1
\leq x_1 \leq \ell-1\}$. Assume that there is a vertex $v\in S\cap
Q_2((0,1))$ (analogous for $v\in S\cap Q_4((1,0))$ by rotating
the situation). Since all the pairs in the same row than
$\{(0,1),(1,0)\}$ are resolved by $v$ and $S$ is minimal, there
is no other vertex of $S$ associated to pairs in such row and so
$|S\setminus \{\alpha\}|\leq 2(\ell-2)+1$ which leads to $|S|\leq
2(\ell-2)+1+1=2\ell-2$, the expected bound.

\end{proof}

Consider now the graph $H_{\ell,m}$ obtained from $G_{\ell}$ by
attaching a set of $m\geq 2$ pendant
 vertices $\{\alpha_1,...,\alpha_m\}$  at $(0,0)$.

\

\noindent \emph{Claim 2.
 $\dime(H_{\ell,{\it m}})={\it m}+1$ and
 $\dime^+(H_{\ell,m})=m+2\ell-4$.}

\begin{proof}
As it was said before, the set $\{(0,0),(\ell-1,0)\}$ is a metric
basis of $G_{\ell}$ \cite{landmarks}. Thus, it can be easily
checked that the set $\{\alpha_1,...,\alpha_m,(\ell-1,0)\}$ is a
resolving set of $H_{\ell,{\it m}}$ which gives
$\dime(H_{\ell,{\it m}})\leq {\it m}+1$. To prove that
$\dime(H_{\ell,{\it m}})\geq {\it m}+1$ it suffices to show that
$|S|\geq m+1$ for every metric basis $S$.

A metric basis $S$ has to contain all the pendant vertices but at
most one. Suppose that $\{\alpha_1,...,\alpha_{m-1}\}\subset S$
and $\alpha_m\notin S$ (if $\{\alpha_1,...,\alpha_{m}\}\subset S$
the result clearly follows). Since no pendant vertex resolves the
pair $\{(0,1),(1,0)\}$ then there is a vertex, say  $u\in
R((0,1),(1,0))=Q_2((0,1)) \cup Q_4((1,0))$. But either the pair
$\{\alpha_m,(1,0)\}$ or the pair $\{\alpha_m,(0,1)\}$ is not
resolved by any vertex in the set
$\{\alpha_1,...,\alpha_{m-1},u\}$ and so
 $|S|\geq m+1$.

Mimicking the proof of Claim 1, only replacing $\alpha$ by
$\alpha_1,\ldots ,\alpha_{m-1}$ (compare Figures
\ref{metricbasis}(a) and \ref{metricbasis}(b)) it is proved that
$\dime^+(H_{\ell,m})=m+2\ell-4$. We omit it for the sake of
brevity.
\end{proof}

\begin{figure}
\begin{center}
\begin{tabular}{ccc}
\includegraphics[width=0.3\textwidth]{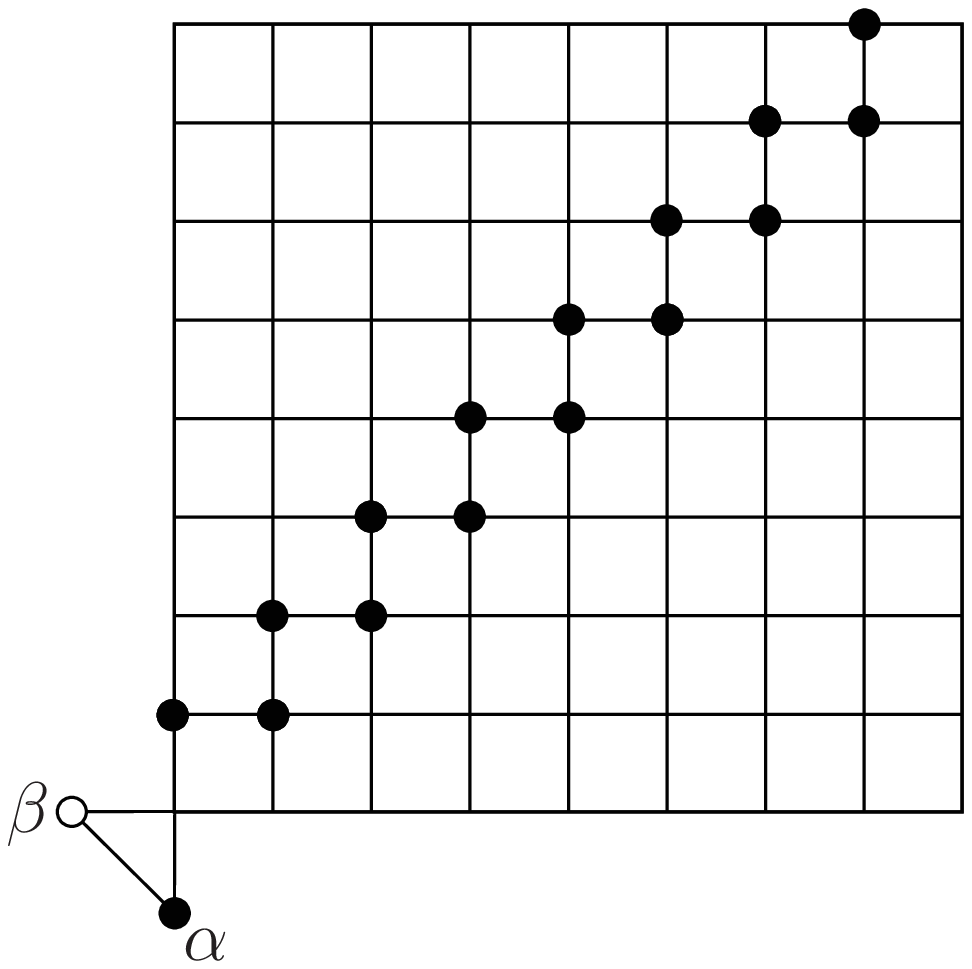}
&&
\includegraphics[width=0.33\textwidth]{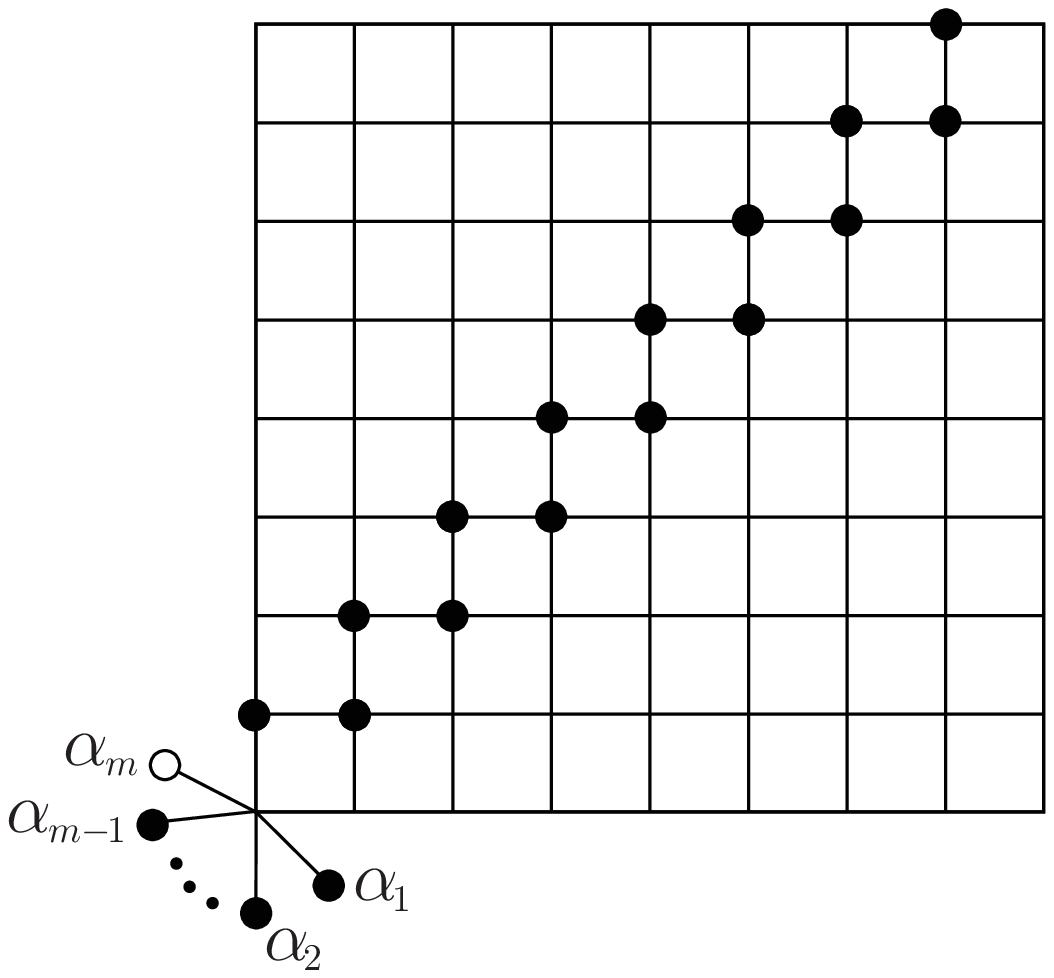}\\
(a) &   &  (b) \\
\end{tabular}
\end{center}
\caption{(a) A minimal resolving set of $H_{\ell}$ of size
$2(\ell-1)$, (b) A minimal resolving set of $H_{\ell,m}$ of size
$m+2\ell-4$. }\label{metricbasis}

\end{figure}

Claims 1 and 2 give the connected graph $G$ with $\dime(G) = a$
and $\dime^+(G) = b$ whenever $a=2$ and $b$ even ($G\cong
H_{\ell}$ for $\ell=(b+2)/2$) or $a>2$ and $b-a$ odd ($G\cong
H_{\ell,m}$ for  $\ell=2+(b-a+1)/2$ and $m=a-1$).

In order to obtain the graph $G$ in the remaining cases, we
modify slightly the graphs $H_{\ell}$ and $H_{\ell,m}$  by
removing the set of vertices $\{(x_1,x_2) \, | \, x_1=\ell-1 \}$.
Denote by  $\tilde H_{\ell}$ and $\tilde H_{\ell,m}$ the
resulting graphs. Note that a $(\ell-2)\times (\ell-1)$ grid, say
$\tilde{G}_{\ell}$, plays now the role of $G_{\ell}$ but all the
tools developed above can also be applied in this case. Hence,
one can follow the proofs of Claims 1 and 2 to compute the metric
dimension and the upper dimension of $\tilde H_{\ell}$ and
$\tilde H_{\ell,m}$. There are only three changes: (1) take the
set $\{(0,0), (\ell-2,0) \}$ as a metric basis of
$\tilde{G}_{\ell}$; (2) remove the vertex $(\ell-2,\ell-1)$ from
$S$ obtaining a minimal resolving set of size $2\ell-3$ (for
$\tilde H_{\ell}$) or $2\ell+m-5$ (for $\tilde H_{\ell,m}$); (3)
apply the column version of Lemma \ref{lemma3} to obtain
$|S\setminus \{\alpha\}|\leq 2(\ell-2)$ or $|S\setminus
\{\alpha_1, \ldots, \alpha_{m-1} \}|\leq 2(\ell-2)$ which
directly gives $|S|\leq 2\ell-3$ (for $\tilde H_{\ell}$) or
$|S|\leq 2\ell+m-5$ (for $\tilde H_{\ell,m}$). Thus, we have

\

\noindent \emph{Claim 3.  $\dime(\tilde H_{\ell})=2$,
$\dime(\tilde H_{\ell,{\it m}})={\it m}+1$, $\dime^+(\tilde
H_{\ell})=2\ell-3$ and $\dime^+(\tilde H_{\ell,m})=m+2\ell-5$.}

\

It gives the graph $G$ with $\dime(G) = a$ and $\dime^+(G) = b$
whenever $a=2$ and $b$ odd ($G\cong \tilde H_{\ell}$ for
$\ell=2+(b-1)/2$) or $a>2$ and $b-a$ even ($G\cong \tilde
H_{\ell,m}$ for $\ell=3+(b-a)/2$ and $m=a-1$).

\end{proof}

\subsection{Realization of the resolving number}

In Subsection 3.1, we have proved that any pair $a, b$ of integers
 such that $2\leq a \leq b$ is realizable as the metric
dimension and the upper dimension of a certain graph. Modifying
slightly
 the above constructions, one can easily prove that
these two integers are realizable as the metric dimension and the
upper dimension of an infinite family of graphs. It suffices to
replace the vertex $(0,0)$ in $G_{\ell}$ by a path of arbitrary
length. If the resulting graph plays the role of $G_{\ell}$ in
the study developed in the previous subsection, then the metric
dimension and upper dimension are preserved.

Theorem \ref{resfinite} below says that, unlike the metric
dimension and the upper dimension, no integer $a\geq 4$ is
realizable as an infinite family of graphs with resolving number
equal to $a$ (note that the path $P_2$ is the only graph with
resolving number $1$ but there are infinite families of graphs
with resolving number 2 and 3, concretely, odd cycles and paths
(for $a=2$) and even cycles (for $a=3$)). In order to prove this
result, we first relate the resolving number to the diameter of a
graph, which is of independent interest.

\begin{proposition}\label{relationships}
Let $G$ be a graph with diameter $d(G)$, girth $\g(G)$ and
resolving number $\res(G)\geq 3$. If $G$ is not an even cycle
then $d(G)\leq 3\res(G)-5.$
\end{proposition}

\begin{proof}

Let us denote $r=\res (G)$. Suppose  on the contrary that
$d(G)>3\res(G)-5$. Then we can assume that there are two vertices
$u$, $v$ such that $d(u,v)=3r-4=3(r-1)-1$. Consider
 a shortest $u$-$v$ path $P=\{u=u_1, u_2, \ldots, u_{3(r-1)}=v\}$
 and suppose that there is a vertex $w\not\in P$ attached at some
  vertex $u_i$ with $i\neq 1,3(r-1)$ (otherwise it can be easily checked that
  $\{u_1,...,u_{r}\}$ is not a resolving set). Clearly, every vertex
 $u_j\in P$ does not resolve either $\{w,u_{i-1}\}$ or $\{w,u_i\}$ or
 $\{w,u_{i+1}\}$. Indeed, assume $i\leq j$ (analogous for $i> j$).
 By Lemma \ref{landmarks}, $u_j$ does not resolve at least one pair among those formed by the vertices
 $u_{i-1},u_i,u_{i+1},w$. Moreover,
 the pairs $\{u_{i-1},u_i\}$, $\{u_{i-1},u_{i+1}\}$ and $\{u_i,u_{i+1}\}$ are all resolved by $u_j$, since
$P$ is a shortest path. Thus, one pair among
 $\{w,u_{i-1}\}$, $\{w,u_i\}$, $\{w,u_{i+1}\}$ is not resolved by $u_j$.

Consider now the sets $A=\{u_j\in P: d(u_j,w)=d(u_j,u_{i-1})\}$,
$B=\{u_j\in P: d(u_j,w)=d(u_j,u_{i})\}$ and $C=\{u_j\in P:
d(u_j,w)=d(u_j,u_{i+1})\}$. Since these sets are not resolving
sets of $G$, then $|A|,|B|,|C|\leq r-1$. Furthermore, $A\cup B
\cup C=P$ and $|P|=3(r -1)$ and so $|A|=|B|=|C|=r-1$ which
implies that $A$, $B$ and $C$ are pairwise disjoint but $u_i\in
A\cap C$; a contradiction.
\end{proof}

Observe that when $\res(G)\leq 2$ or $G$ is an even cycle,
Proposition~\ref{relationships} does not hold. It suffices to
consider the path $P_2$ (for $\res(G)=1$), an odd cycle of length
at least 5 (for $\res(G)=2$) and an even cycle of length at least
6 (for $\res(G)=3$).

\begin{theorem}\label{resfinite}
For every integer $a\geq 4$, the set of graphs with resolving number $a$ is finite.
\end{theorem}

\begin{proof}
A graph $G$ of order $n$, diameter $d(G)$ and metric dimension
$\dime (G) $ satisfies the following relation \cite{landmarks}:
$$n\leq d(G)^{\dime (G)}+\dime (G).$$ Since $\dime (G)\leq
\res(G)$ then $$n\leq  d(G)^{\res(G)}+\res(G)$$ and Proposition
\ref{relationships} gives $$n\leq
(3\res(G)-5)^{\res(G)}+\res(G)=(3a-5)^a+a.$$ This upper bound for
$n$ depends only on the value of $a$ and so the result follows.

\end{proof}

\section{Concluding remarks and open questions}

In this paper, we have characterized  the randomly
$k$-dimensional graphs. Our proof is based on combinatorial
arguments which let us avoid the brute force casuistic analysis.
Moreover, we have also proved  in the affirmative a conjecture
posed by Chartrand et al. \cite{upper} claiming that every pair
$a,b$ of integers with $2\leq a \leq b$ is realizable as the
metric dimension and the upper dimension, respectively, of some
connected graph. We have concluded the paper showing that,
surprisingly, no integer $a\geq 4$ is realizable as the resolving
number of an infinite family of graphs.

It would be interesting to study the realization of triples of
integers $a$, $b$, $c$ as the metric dimension, the upper
dimension and the resolving number, respectively, of some
connected graph. Also, the question of bounding the size of the
set of graphs (maybe restricting to specific families) with given
resolving number $a$ remains open.

\end{document}